\documentclass[reqno,centertags,12pt]{amsart}
\usepackage[letterpaper,margin=1.2in]{geometry}
\usepackage{amsmath,amsthm,amsfonts,amssymb,enumerate}
\usepackage[bookmarksopen=false,final]{hyperref}

\newcommand{\no}{\notag}
\newcommand{\lb}{\label}
\newcommand{\bb}{\mathbb}
\newcommand{\cc}{\mathcal}

\newcommand{\ol}{\overline}
\newcommand{\bs}{\backslash}
\newcommand{\bbR}{\mathbb{R}}
\newcommand{\bbC}{\mathbb{C}}
\newcommand{\bbD}{\mathbb{D}}

\newcommand{\De}{\Delta}

\newcommand{\Om}{\Omega}

\newcommand{\eps}{\varepsilon}

\newcommand{\cvh}{\mathrm{cvh}}
\newcommand{\dist}{\mathrm{dist}}
\newcommand{\diam}{\mathrm{diam}}
\newcommand{\E}{\ensuremath{\mathsf{E}}}
\newcommand{\ca}{\mathrm{Cap}}
\newcommand{\W}{\mathcal{W}}
\newcommand{\PW}{\mathcal{PW}}
\newcommand{\IR}{\mathcal{IR}}
\newcommand{\ir}{\mathrm{ir}}
\newcommand{\reg}{\mathrm{reg}}

\newtheorem{theorem}{Theorem}
\newtheorem{lemma}[theorem]{Lemma}

\newtheorem{corollary}[theorem]{Corollary}
\theoremstyle{definition}

\newtheorem{example}[theorem]{Example}
\theoremstyle{remark}
\newtheorem*{remark}{Remark}
\numberwithin{equation}{section}
\numberwithin{theorem}{section}

\begin{document}
\title[Widom factors on semi-regular subsets of $\mathbb{R}$] {Widom factors for Chebyshev and residual polynomials on semi-regular subsets of $\mathbb{R}$}
	
\author{Robert Dukes}
\address{Department of Mathematics and Statistics, University of New Mexico, Albuquerque, NM 87131, USA}
\email{dukesr@unm.edu}

\author{Maxim Zinchenko}
\address{Department of Mathematics and Statistics, University of New Mexico, Albuquerque, NM 87131, USA}
\email{maxim@math.unm.edu}
\thanks{\footnotesize M.Z. is supported in part by the Simons Foundation Grant MP--TSM--00002651.}


\begin{abstract}
We study $L^\infty$ Widom factors for Chebyshev and residual polynomials on compact non-polar subsets of the real line that need not be regular in the sense of potential theory and, in particular, on sets with isolated points.  We first show that boundedness of the Widom factors is independent of the normalization point $x_*\in\overline{\mathbb{R}}\backslash\E$.  For semi-regular sets, meaning that the set of regular points $\E^\reg$ is closed, we introduce the irregularity coefficient
$\IR(\E,x_*)=\sum_{x\in\E\bs\E^\reg}G_\E(x,x_*)$, which measures the contribution of irregular boundary points to the extremal polynomial problem. Our upper bound is
\[
    \sup_{n\ge1}\W_n(\E,x_*)
    \le
    2\exp\bigl[\PW(\E^\reg,x_*)+\IR(\E,x_*)\bigr],
\]
while the complementary lower bound is
\[
    \liminf_{n\to\infty}\W_n(\E,x_*)
    \ge
    2\exp\bigl[\IR(\E,x_*)\bigr].
\]
Consequently, for semi-regular Parreau--Widom sets the Widom factors are bounded if and only if $\IR(\E,x_*)<\infty$.  When the regular part is a finite union of intervals, this condition is equivalent to a geometric square-root summability condition on the irregular points.  In the special case $\E=[a,b]\cup\{x_k\}_{k\ge1}$, the Widom factors have the exact, possibly infinite, limit $2\exp[\IR(\E,x_*)]$. When this limit is finite, we obtain the corresponding Szeg\H{o}--Widom asymptotics for the normalized extremal polynomials.
\end{abstract}

\date{\today}
\maketitle

\section{Introduction}

In this work we consider compact sets $\E\subset\bbR$ of positive logarithmic capacity and denote the supremum norm on $\E$ by $\|\cdot\|_\E$.  The $n$-th Chebyshev polynomial $T_n$ on $\E$ is the unique monic polynomial of degree $n$ that minimizes $\|\cdot\|_\E$. The associated Widom factor is
\begin{align}\label{Wn-cheb}
    \W_n(\E)=\frac{\|T_n\|_\E}{\ca(\E)^n}.
\end{align}
By Szeg\H{o}'s root asymptotic theorem, $\|T_n\|_\E^{1/n}\to\ca(\E)$, so the normalization by $\ca(\E)^n$ removes the leading exponential behavior of the Chebyshev norm and leaves a quantity that captures the subexponential behavior. In this paper we investigate how the asymptotic behavior of the Widom factors is governed by the potential-theoretic geometry of $\E$, with particular emphasis on the contribution of irregular points.

We also consider the corresponding residual extremal problem. Given $x_*\in\bbR\bs\E$, let $T_{n,x_*}$ be the unique polynomial that minimizes $\|P\|_\E$ among all polynomials $P$ of degree at most $n$ satisfying $P(x_*)=1$. The residual polynomials satisfy the root asymptotics $\|T_{n,x_*}\|_\E^{1/n}\to\ca(\E,x_*)$, see e.g., \cite[Thm.2.4(a)]{CSZ5}. Here $\ca(\E,x_*)$ is defined by
\begin{align}\label{Cap-res}
    \ca(\E,x_*)
    =
    \exp[-G_\E(x_*)],
\end{align}
and $G_\E(z)=G_\E(z,\infty)$ denotes the Green function with pole at infinity.

Following the convention of \cite{CSZ5}, we define the residual Widom factor by
\begin{align}\label{Wn-res}
    \W_n(\E,x_*)
    =
    \|T_{n,x_*}\|_\E \,
    [\ca(\E,x_*)^n+\ca(\E,x_*)^{-n}].
\end{align}
The Widom factors defined in \eqref{Wn-cheb} and \eqref{Wn-res} are compatible in the sense that
\begin{align*}
    \lim_{x_*\to\infty}\W_n(\E,x_*)=\W_n(\E).
\end{align*}
This continuity result follows from \cite[Prop.~2.2]{CSZ6}, after translating between the normalizations used there and the convention in \eqref{Wn-res}. We therefore set
\[
    \W_n(\E,\infty)=\W_n(\E),\quad
    \ca(\E,\infty)=\ca(\E),\quad
    T_{n,\infty}=T_n.
\]
This convention allows the Chebyshev and residual problems to be treated simultaneously by considering reference points $x_*\in\ol{\bbR}\bs\E$, where $\ol{\bbR}=\bbR\cup\{\infty\}$.

We are interested in upper and lower bounds on the Widom factors for sets that are not regular in the sense of potential theory. Without additional assumptions on $\E$, the Widom factors may be unbounded even if the set is regular. In fact, on sufficiently thin Cantor-type compact subsets of $\bbR$, they can exhibit arbitrary subexponential growth \cite{GonHat15}. On the other hand, for regular Parreau--Widom sets, which include many thick Cantor-type sets such as homogeneous sets in the sense of Carleson, the Widom factors are uniformly bounded. Uniform upper bounds on Widom factors are called Totik--Widom bounds.

For regular Parreau--Widom compact sets $\E\subset\bbR$ and $x_*\in\ol{\bbR}\bs\E$, the sharp Totik--Widom bound
\begin{align}\label{UB-reg}
    \sup_{n\ge1}\W_n(\E,x_*)\le 2\exp\bigl[\PW(\E,x_*)\bigr]
\end{align}
was proved in \cite{CSZ1} for Chebyshev polynomials and in \cite{CSZ5} for residual polynomials. The Parreau--Widom constant is defined by
\begin{align*}
    \PW(\E,x_*)
    =
    \sum_{c\in\cc C(\E,x_*)}G_\E(c,x_*),
\end{align*}
where $\cc C(\E,x_*)$ denotes the set of critical points of the Green function $G_\E(\cdot,x_*)$ in $\ol{\bbC}\bs(\E\cup\{x_*\})$.
We note that, for regular compact sets, the numerical value of $\PW(\E,x_*)$ depends on $x_*\in\ol{\bbR}\bs\E$, but the finiteness of the sum does not, see Corollary~\ref{IR-PW-inv}. 
Regular compact sets satisfying $\PW(\E,x_*)<\infty$ for one, and hence for all, normalization points $x_*\in\ol{\bbR}\bs\E$ are called \emph{Parreau--Widom sets}.

It was proved in \cite{CSYZ2} that, for generic regular compact sets $\E\subset\bbR$, the Chebyshev Widom factors $\W_n(\E,\infty)$ are bounded if and only if $\E$ is a Parreau--Widom set. We extend this equivalence to residual polynomials, under the same genericity hypothesis in Corollary~\ref{Wn-PW-eqv-cor}, by showing that boundedness of the Widom factors $\W_n(\E,x_*)$ is independent of the choice of $x_*\in\ol{\bbR}\bs\E$, see Theorem~\ref{Wn-ind-thm}. Whether the equivalence holds for all regular compact subsets of $\bbR$ remains open.

In Section~\ref{Sec3} we extend the Totik--Widom bound \eqref{UB-reg} to compact non-polar sets that are not necessarily regular. We write $\E^\reg$ and $\E^\ir$ for the regular and irregular points of $\E$, respectively. We call a compact non-polar set $\E$ \emph{semi-regular} if $\E^\reg$ is closed. This condition is weaker than regularity and, in particular, it allows $\E$ to have isolated points. By Kellogg's theorem \cite[Thm.~4.2.5]{Ran95}, the set of irregular points $\E^\ir$ is polar.  Hence replacing a semi-regular set $\E$ by $\E^\reg$ does not change the Green function in the complement, but it can change the extremal polynomial problem because the supremum norm $\|\cdot\|_\E$ is taken over all of $\E$, including the irregular points.

The contribution of the irregular points to the size of the Widom factors is measured by the irregularity coefficient
\begin{align*}
    \IR(\E,x_*)=\sum_{x\in\E^\ir}G_\E(x,x_*),
\end{align*}
where boundary values of the Green function at points of $\E$ are understood as upper boundary limits, as recalled in Section~\ref{Sec2.1}.  Although the numerical value of $\IR(\E,x_*)$ depends on $x_*$, its finiteness is independent of the choice of $x_*\in\ol{\bbR}\bs\E$, see Corollary~\ref{IR-PW-inv}.

One of the main results of this paper is a Totik--Widom upper bound for semi-regular sets. If $\E$ is semi-regular, then $\E^\reg$ is compact and regular, and the Parreau--Widom constant $\PW(\E^\reg,x_*)$ is well defined. We prove
\begin{align}\label{UB-intro}
    \sup_{n\ge1}\W_n(\E,x_*)\le
    2\exp\bigl[\PW(\E^\reg,x_*)+\IR(\E,x_*)\bigr].
\end{align}
This shows that irregular points enter the upper bound through a separate additive term in the exponent. The regular part contributes the usual Parreau--Widom constant, while the polar irregular part contributes the irregularity coefficient. When $\E$ is regular, $\IR(\E,x_*)=0$ and $\E^\reg=\E$, so \eqref{UB-intro} reduces to \eqref{UB-reg}.

Unlike the upper bound, the basic lower bound of Schiefermayr \cite{Sch08, Sch11, Sch17}, see also \cite{CSZ1, CSZ5}, holds for arbitrary compact sets $\E\subset\bbR$ of positive capacity:
\begin{align}\label{LB}
    \inf_{n\ge1}\W_n(\E,x_*)\ge 2.
\end{align}
In Section~\ref{Sec4} we derive the following asymptotic improvement for compact semi-regular sets:
\begin{align}\label{ALB-intro}
    \liminf_{n\to\infty}\W_n(\E,x_*)
    \ge
    2\exp\bigl[\IR(\E,x_*)\bigr].
\end{align}
This shows that the contribution of irregular points captured by the $\exp[\IR(\E,x_*)]$ factor in \eqref{UB-intro} is asymptotically sharp.

Combining \eqref{UB-intro} and \eqref{ALB-intro} gives the principal boundedness criterion of the paper: if $\E$ is a semi-regular Parreau--Widom set, then
\begin{align*}
    \sup_{n\ge1}\W_n(\E,x_*)<\infty
    \quad\Longleftrightarrow\quad
    \IR(\E,x_*)<\infty.
\end{align*}
Thus, within the semi-regular Parreau--Widom class, boundedness of Widom factors is governed exactly by the irregular part of the set.

For compact sets $\E\subset\bbR$ whose regular part $\E^\reg$ is a finite union of intervals, the condition $\IR(\E,x_*)<\infty$ admits a geometric reformulation:
\begin{align}\label{IR-geom}
    \sum_{x\in\E^\ir}\dist(x,\E^\reg)^{1/2}<\infty.
\end{align}
In this case $\E^\reg$ is automatically a regular Parreau--Widom set, and hence the Widom factors $\W_n(\E,x_*)$ are bounded if and only if \eqref{IR-geom} holds.

Finally, in Section~\ref{Sec5}, we consider sets of the form $\E=[a,b]\cup\{x_k\}_{k\ge1}\subset\bbR$. In this case the upper and lower bounds \eqref{UB-intro} and \eqref{ALB-intro} match, yielding a possibly infinite limit for the Widom factors,
\begin{align*}
    \lim_{n\to\infty}\W_n(\E,x_*)
    =
    2\exp\big[\IR(\E,x_*)\big].
\end{align*}
When this limit is finite, equivalently when $\sum_k\dist(x_k,[a,b])^{1/2}<\infty$, we also prove Szeg\H{o}--Widom asymptotics for the normalized extremal polynomials. In this asymptotics the isolated irregular points enter through an explicit Blaschke product in the conformal image of $\ol{\bbC}\bs[a,b]$.

We also note that similar phenomena occur in the $L^2$ theory of orthogonal polynomials on the real line. For measures with an interval essential support and a denumerable set of mass points outside the interval, Simon and Zlato\v{s} \cite{SimZla03} related boundedness of the corresponding $L^2$-Widom factors to a Blaschke condition on the mass points and a Szeg\H{o} condition on the absolutely continuous part of the measure; see also \cite{Sim11} for a detailed account. The finite-gap and Parreau--Widom analogues were studied in \cite{CSZ11} and \cite{Chr12}, respectively. In the interval case with mass points, Szeg\H{o}--Widom asymptotics for orthogonal polynomials was proved by Peherstorfer and Yuditskii \cite{PehYud01} under the Blaschke and Szeg\H{o} conditions. Further polynomial asymptotics in the presence of mass points were studied in \cite{PehYud03,PehYud06,Chr19}.


\section{Preliminaries}\label{Sec2}

In this section we collect the potential-theoretic and extremal-polynomial facts used throughout the paper. 

\subsection{Green functions and semi-regular sets}\label{Sec2.1}

Let $\E\subset\bbR$ be a compact non-polar set and consider the domain $\Omega_\E=\ol{\bbC}\bs\E$. For $x_*\in\Omega_\E$ we denote by $G_\E(z,x_*)$ the Green function of $\Omega_\E$ with pole at $x_*$. We write $G_\E(z)=G_\E(z,\infty)$ for the Green function with pole at infinity. The logarithmic capacity $\ca(\E)$ is determined by the asymptotics of $G_\E(z)$ at infinity,
\begin{align}\label{Cap-cheb}
    G_\E(z)=\log|z|-\log\ca(\E)+o(1),
    \quad z\to\infty.
\end{align}
For $x\in\E$, the boundary value $G_\E(x,x_*)$ will always be understood as the upper boundary value
\[
    G_\E(x,x_*)=\limsup_{\Omega_\E\ni z\to x}G_\E(z,x_*).
\]
These boundary values are finite at every point $x\in\E$. For $x_*=\infty$, this follows from the logarithmic-potential representation of $G_\E(z)$ since the logarithmic potential of a finite compactly supported measure is bounded from below on bounded sets. The case of finite $x_*$ follows by applying a M\"obius transformation sending $x_*$ to $\infty$ and using conformal invariance of Green functions.  Thus, the extended Green function $G_\E(\cdot,x_*)$ is finite, nonnegative, and upper semicontinuous on $\ol{\bbC}\bs\{x_*\}$.
Throughout the paper, Green functions are understood to be extended in this way.

By the Green-function characterization of regular boundary points \cite[Thm.~4.4.9]{Ran95}, a point $x\in\E$ is regular for the Dirichlet problem in $\Omega_\E$ if $G_\E(x,x_*)=0$ and $x\in\E$ is irregular if $G_\E(x,x_*)>0$. This characterization is independent of the choice of $x_*$. We denote the sets of regular and irregular points by $\E^\reg$ and $\E^\ir$, respectively. We call a compact non-polar set $\E\subset\bbR$ \emph{semi-regular} if $\E^\reg$ is a closed set.

For a compact semi-regular set $\E\subset\bbR$, the set $\E^\reg$ is compact. Since $\E^\ir$ is polar, $\ca(\E)=\ca(\E^\reg)$ and $G_\E(\cdot,x_*)=G_{\E^\reg}(\cdot,x_*)$ on $\Omega_\E$ and hence on $\ol\bbC\bs\{x_*\}$. It follows that $\E^\reg$ is itself a regular compact set and $G_\E(\cdot,x_*)$ is continuous on $\ol\bbC\bs\{x_*\}$.

\subsection{The Parreau--Widom constant and the irregularity coefficient}\label{Sec2.2}

By a gap of $\E$ we mean a connected component of $\ol{\bbR}\bs\E$. We call the component containing infinity the exterior gap of $\E$.  For a regular compact set $\E\subset\bbR$, let $\cc C(\E,x_*)$ denote the set of critical points of $G_\E(\cdot,x_*)$ in $\ol{\bbC}\bs(\E\cup\{x_*\})$. For real compact sets and real poles, these critical points lie in the gaps of $\E$, there is no critical point in the gap containing $x_*$, and every other gap contains exactly one critical point. We refer to \cite{CSZ1,CSZ5} for this standard description.

For regular compact sets $\E\subset\bbR$, the Parreau--Widom constant is defined by
\begin{align}\label{PW-reg}
    \PW(\E,x_*)
    =
    \sum_{c\in\cc C(\E,x_*)}G_\E(c,x_*),
\end{align}
where we use the convention that $\PW(\E,x_*)=0$ if $\cc C(\E,x_*)$ is empty.
For compact semi-regular sets $\E\subset\bbR$ we define $\PW(\E,x_*)$ as the Parreau--Widom constant of the subset of regular points $\E^\reg$,
\begin{align}\label{PW-def}
    \PW(\E,x_*)=\PW(\E^\reg,x_*),
\end{align}
which is natural since $G_\E(\cdot,x_*)$ and $G_{\E^\reg}(\cdot,x_*)$ coincide on $\ol{\bbC}\bs\{x_*\}$. We call a compact semi-regular set $\E$ a \emph{Parreau--Widom set} if
$\PW(\E,x_*)$ is finite for one, and hence by Corollary~\ref{IR-PW-inv}, for every
$x_*\in\ol{\bbR}\bs\E$.

The \emph{irregularity coefficient} is defined by
\begin{align}\label{IR-def}
    \IR(\E,x_*)=
    \sum_{x\in\E^\ir}G_\E(x,x_*),
\end{align}
where we use the convention that $\IR(\E,x_*)=0$ if $\E^\ir$ is empty and $\IR(\E,x_*)=\infty$ if $\E^\ir$ is uncountable.
\begin{remark}
The definition of the irregularity coefficient is equivalent to
\[
    \IR(\E,x_*)=\sup_F\sum_{x\in F}G_\E(x,x_*),
\]
where the supremum is over all finite sets $F\subset\E^\ir$. Indeed, if $\E^\ir$ is countable this is the usual equality between a positive series and the supremum of its finite partial sums. If $\E^\ir$ is uncountable, then, since $G_\E(x,x_*)>0$ for every $x\in\E^\ir$, one of the sets $\{x\in\E^\ir:\,G_\E(x,x_*)>1/k\}$ must be infinite. This implies that the finite sums can be arbitrarily large and the supremum is therefore infinite.
\end{remark}

Throughout the paper we will use the convention $\exp(\infty)=\infty$.

\medskip

The following lemma shows that the Green functions corresponding to different normalization points are comparable near $\E$.

\begin{lemma}\label{G-compar-lem}
Let $\E\subset\bbR$ be a compact non-polar set and $x_1,x_2\in\ol{\bbR}\bs\E$. Then there is a neighborhood $U$ of $\E$ and a constant $C>1$ such that, for all $z\in U$,
\begin{align}\label{G-compar}
    C^{-1}G_\E(z,x_1)\le G_\E(z,x_2)\le C G_\E(z,x_1).
\end{align}
\end{lemma}

\begin{proof}
Let $U$ be an $\eps$-neighborhood of $\E$, with $\eps>0$ small enough that $\overline U$ does not contain $x_1$ or $x_2$. 
Then $\Omega_\E\bs\overline U$ is an open connected set containing $x_1,x_2$.

For each fixed $z\in U\bs\E$, the function $w\mapsto G_\E(z,w)$ is positive and harmonic in $\Omega_\E\bs\overline U$. By Harnack's inequality, applied in the $w$-variable on the compact set $\{x_1,x_2\}\subset\Omega_\E\bs\overline U$, there is a constant $C>1$, independent of $z\in U\bs\E$, such that \eqref{G-compar} holds. Then taking upper limits as $U\bs\E\ni z\to x\in\E$ gives the same comparison for the boundary values at all points of $\E$.
\end{proof}

\begin{corollary}\label{IR-PW-inv}
Let $\E\subset\bbR$ be a compact non-polar set. The finiteness of $\IR(\E,x_*)$ is independent of the choice of $x_*\in\ol{\bbR}\bs\E$. If $\E$ is semi-regular, then the finiteness of $\PW(\E,x_*)$ is also independent of the choice of $x_*\in\ol{\bbR}\bs\E$.
\end{corollary}

\begin{proof}
The assertion for $\IR$ follows directly from Lemma~\ref{G-compar-lem}, since all irregular points lie in $\E$ and the boundary values of the two Green functions are comparable there.

For $\PW$ we apply a similar argument to $\E^\reg$. Let $x_1,x_2\in\ol\bbR\bs\E$ be two normalization points and $U$ be a small $\eps$-neighborhood of $\E^\reg$ from Lemma~\ref{G-compar-lem}. Only finitely many gaps of $\E^\reg$ are not contained in $U$, and hence such gaps do not affect convergence of the Parreau--Widom sum.

Let $L\subset U$ be a gap of $\E^\reg$. Then $x_1,x_2\not\in L$ since the same holds for $U$. Let $c_j$ be the critical point of $G_{\E^\reg}(\cdot,x_j)$ in $L$, $j=1,2$. By \eqref{G-compar} and by maximality of the critical value in the gap,
\[
    G_{\E^\reg}(c_1,x_1)
    \le
    C G_{\E^\reg}(c_1,x_2)
    \le
    C G_{\E^\reg}(c_2,x_2).
\]
After summing over all such gaps, and adding the finite contribution from the gaps not contained in $U$, convergence of $\PW(\E^\reg,x_2)$ implies convergence of $\PW(\E^\reg,x_1)$. Interchanging $x_1$ and $x_2$ gives the converse.
\end{proof}

\subsection{Extremal polynomials and inverse-image sets}\label{Sec2.3}

Here we recall several facts about Chebyshev and residual extremal polynomials proved in \cite{CSZ1,CSZ5}. Given a compact non-polar set $\E\subset\bbR$, a normalization point $x_*\in\ol{\bbR}\bs\E$, and the corresponding extremal polynomial $T_{n,x_*}$, let
\[
    t_{n,x_*}(\E)=\|T_{n,x_*}\|_\E,
    \quad
    d_n=\deg T_{n,x_*}.
\]
The extremal norm is monotone with respect to set inclusion, that is, if $K\subset\E$ and $x_*\in\ol{\bbR}\bs\E$, then $t_{n,x_*}(\E)\ge t_{n,x_*}(K)$ since every admissible polynomial for $\E$ is admissible for $K$ and $\|P\|_K\le\|P\|_\E$.

The extremal polynomials satisfy the alternation theorem, which yields the properties recalled below.

The extremal polynomials for normalization points $x_1$, $x_2$ lying in the same gap of $\E$ are scalar multiples of one another, $T_{n,x_1}=c_{n,x_1,x_2} T_{n,x_2}$ for some $z$-independent constant $c_{n,x_1,x_2}$. The degrees satisfy $n-1\le d_n\le n$. In the Chebyshev case $d_n=n$ and hence the same holds for all $x_*$ in the exterior gap of $\E$. The exceptional possibility $d_n=0$ can occur only when $n=1$, in which case $T_{1,x_*}\equiv 1$.

For $d_n\ge1$ define
\begin{align*}
    \E_n=T_{n,x_*}^{-1}\big([-t_{n,x_*}(\E),t_{n,x_*}(\E)]\big).
\end{align*}
Then $\E\subset\E_n\subset\bbR$ and $\E_n$ is a finite union of $d_n$ compact intervals, called bands, such that each band is mapped bijectively by $T_{n,x_*}$ onto $[-t_{n,x_*}(\E),t_{n,x_*}(\E)]$. In particular, all zeros of $T_{n,x_*}$ lie in $\E_n$ and are real and simple.
The rescaling property of the extremal polynomials implies that $\E_n$ is
independent of the choice of $x_*$ within a given gap of $\E$.

Each gap of $\E$ intersects at most one band of $\E_n$. The gap of $\E$ containing the reference point $x_*$ is disjoint from $\E_n$ and if $d_n=n-1$ the exterior gap is also disjoint from $\E_n$. Consequently, $T_{n,x_*}$ has at most one zero in each gap of $\E$, no zeros in the gap containing $x_*$ and if $d_n=n-1$ also no zeros in the exterior gap of $\E$.

Let $\mu_n$ denote the equilibrium measure of $\E_n$. Each band of $\E_n$ has $\mu_n$-measure $1/d_n$. Consequently, if $K$ is a gap of $\E$ that intersects $\E_n$, then
\begin{align}\label{gap-meas}
    \mu_n(K)\le \frac1{d_n}.
\end{align}

The norms of the extremal polynomials can be expressed in terms of $\E_n$:
\begin{align}
    \|T_n\|_\E &=2\ca(\E_n)^n,
    \label{norm-ca-cheb}
    \\
    \|T_{n,x_*}\|_\E &=\frac{1}{\cosh(d_nG_{\E_n}(x_*))},
    \quad x_*\neq\infty.
    \label{norm-ca-res}
\end{align}
Then by \eqref{Wn-cheb}, \eqref{Wn-res}, and \eqref{Cap-res},
\begin{align}
    \W_n(\E,\infty)
    &=2\frac{\ca(\E_n)^n}{\ca(\E)^n},
    \label{WnEn-cheb}
    \\
    \begin{split}
    \W_n(\E,x_*)
    &=2\frac{\exp[nG_\E(x_*)]+\exp[-nG_\E(x_*)]} {\exp[d_nG_{\E_n}(x_*)]+\exp[-d_nG_{\E_n}(x_*)]} \\
    &=2\frac{\exp[nG_\E(x_*)]}{\exp[d_nG_{\E_n}(x_*)]} \,
      \frac{1+\exp[-2nG_\E(x_*)]}
           {1+\exp[-2d_nG_{\E_n}(x_*)]}, \quad x_*\neq\infty.
    \end{split}
    \label{WnEn-res}
\end{align}

Similarly the normalized extremal polynomials $T_{n,x_*}(z)/\|T_{n,x_*}\|_\E$ can be expressed in terms of the complex Green functions of $\Om_{\E_n}$, see \cite[Eq.~(2.4)]{CSZ1} for the Chebyshev case and \cite[Eq.~(3.9)]{CSZ5} for residual polynomials. For our purposes we will only need the following modulus consequence,
\begin{align}\label{Tn-pointwise}
    \frac{|T_{n,x_*}(z)|}{\|T_{n,x_*}\|_\E}
    \le
    \frac12\bigl(e^{d_nG_{\E_n}(z)}+e^{-d_nG_{\E_n}(z)}\bigr),
    \quad z\in\Omega_{\E_n}\bs\{\infty\}.
\end{align}

To derive our result below, we will use the difference of Green functions,
\begin{align}\label{De-def}
    \De_{\E,n}(z)=G_\E(z)-G_{\E_n}(z),
    \quad z\in\Omega_{\E_n}\bs\{\infty\}.
\end{align}
The function $\De_{\E,n}$ extends harmonically to $z=\infty$ with the value
\begin{align}
    \De_{\E,n}(\infty)
    =
    \lim_{z\to\infty}\bigl(G_\E(z)-G_{\E_n}(z)\bigr)
    =
    \log\frac{\ca(\E_n)}{\ca(\E)}.
\end{align}
Since $\E\subset\E_n$, monotonicity of Green functions gives $\De_{\E,n}\ge0$ on $\Omega_{\E_n}$.
We will use the representation
\begin{align}\label{De-rep-z}
    \De_{\E,n}(z)
    =
    \int_{\E_n\bs\E^\reg}G_\E(x,z)\,d\mu_n(x)
    =
    \int_{\E_n\bs\E}G_\E(x,z)\,d\mu_n(x),
    \quad z\in\Omega_{\E_n},
\end{align}
see \cite[Prop.~4.3]{CSYZ2}, \cite[Lem.~3.6]{ELY24}.
Indeed, for fixed finite $z\in\Omega_{\E_n}$, the function
\[
    w\mapsto G_\E(w,z)-G_{\E_n}(w,z)
\]
extends harmonically to $\Omega_{\E_n}$ due to cancellation of singularities at $w=z$ and has boundary values $G_\E(\cdot,z)$ quasi-everywhere on $\E_n$. Thus it is the Perron solution of the Dirichlet problem in $\Omega_{\E_n}$ with boundary data $G_\E(\cdot,z)$. Since harmonic measure at infinity for $\Omega_{\E_n}$ is the equilibrium measure $\mu_n$, evaluating this solution at infinity gives
\[
    G_\E(\infty,z)-G_{\E_n}(\infty,z)
    =
    \int_{\E_n}G_\E(x,z)\,d\mu_n(x).
\]
By symmetry of Green functions, the left hand side is $\De_{\E,n}(z)$.
Since $G_\E(\cdot,z)=0$ on $\E^\reg$, this gives the first equality in \eqref{De-rep-z}. The second equality follows since $\E^\ir$ is polar and equilibrium measures do not charge polar sets. The case $z=\infty$ follows by harmonic extension.

In particular, for every $x_*\in\ol{\bbR}\bs\E_n$, we have
\begin{align}\label{De-rep}
    \De_{\E,n}(x_*)
    =
    \int_{\E_n\bs\E^\reg}G_\E(x,x_*)\,d\mu_n(x)
    =
    \int_{\E_n\bs\E}G_\E(x,x_*)\,d\mu_n(x).
\end{align}
On the other hand, \eqref{De-def} together
with the definitions of $\ca(\cdot,x_*)$ gives
\begin{align*}
    \De_{\E,n}(x_*)
    =
    \log\frac{\ca(\E_n,x_*)}{\ca(\E,x_*)},
    \qquad x_*\in\ol{\bbR}\bs\E_n.
\end{align*}
Then using \eqref{WnEn-cheb}, \eqref{WnEn-res}, and \eqref{De-def}, we can express the Widom factors in terms of $\De_{\E,n}$. In the Chebyshev case $x_*=\infty$ we get
\begin{align}\label{WnDe-cheb}
    \W_n(\E,\infty)=2\exp\big[n\De_{\E,n}(\infty)\big],
\end{align}
and in the residual case $x_*\in\bbR\bs\E$ we get
\begin{align}\label{WnDe-res}
    \W_n(\E,x_*)
    =
    2\exp\big[d_n\De_{\E,n}(x_*)+(n-d_n)G_\E(x_*)\big]\,
    \frac{1+\exp[-2nG_\E(x_*)]}{1+\exp[-2d_nG_{\E_n}(x_*)]}.
\end{align}
For later use, note that the last fraction in \eqref{WnDe-res}
tends to one as $n\to\infty$. Indeed, by the root asymptotics $\|T_{n,x_*}\|_\E^{1/n}\to\ca(\E,x_*)$ and the norm
identity \eqref{norm-ca-res},
\[
    \cosh^{1/n}(d_nG_{\E_n}(x_*))
    =
    \|T_{n,x_*}\|_\E^{-1/n}
    \to
    \exp[G_\E(x_*)].
\]
Then, since $\cosh(y)\le e^y\le2\cosh(y)$ for $y\ge0$,
\[
    \exp\left[\frac{d_n}{n}G_{\E_n}(x_*)\right]
    \to
    \exp[G_\E(x_*)].
\]
Thus $(d_n/n)G_{\E_n}(x_*)\to G_\E(x_*)>0$. It follows that both exponential terms in the last fraction in \eqref{WnDe-res} tend to zero, so
\begin{align}\label{WnDe-res-asymp}
    \W_n(\E,x_*)
    =
    2\exp\big[d_n\De_{\E,n}(x_*)+
    (n-d_n)G_\E(x_*)\big](1+o(1)).
\end{align}

Finally, we recall the Bernstein--Walsh inequality. If $P$ is a polynomial of degree at most $m$, then
\begin{align}\label{BW-ineq}
    |P(z)|
    \le
    \|P\|_\E\exp[mG_\E(z)],
    \quad z\in\bbC.
\end{align}

\subsection{Changing the normalization point}

We next show that boundedness of Widom factors does not depend on the normalization point. The result holds for arbitrary compact non-polar subsets of $\bbR$.

\begin{theorem}\label{Wn-ind-thm}
Let $\E\subset\bbR$ be a compact non-polar set. Then boundedness of the Widom factors is independent of the normalization point. More precisely, for any $x_1,x_2\in\ol{\bbR}\bs\E$,
\[
    \sup_{n\ge1}\W_n(\E,x_1)<\infty
    \quad\Longleftrightarrow\quad
    \sup_{n\ge1}\W_n(\E,x_2)<\infty.
\]
\end{theorem}

\begin{proof}
It is enough to compare any two finite normalization points and then compare one finite point outside the convex hull $\cvh(\E)$ of $\E$ with $\infty$.

First let $x_1,x_2\in\bbR\bs\E$ and assume $x_1\neq x_2$. Suppose that $\sup_n\W_n(\E,x_1)<\infty$. Then there is a constant $M$ such that for all $n\ge1$,
\[
    \|T_{n,x_1}\|_\E\le Me^{-nG_\E(x_1)}.
\]
We recall that in the gap $K_0$ of $\E$ containing $x_1$, $T_{n,x_1}$ has no zeros.

Choose a compact interval $I=[\alpha,\beta]\subset\bbR\bs(\E\cup\{x_1\})$ with $x_2\in(\alpha,\beta)$. Then $I$ is contained in a gap of $\E$, and hence $T_{n,x_1}$ has at most one zero in $I$. Let $Q_n$ be $T_{n,x_1}$ with this possible zero in $I$ removed, that is, $Q_n=T_{n,x_1}/S_n$, where $S_n=1$ if $T_{n,x_1}$ has no zero in $I$ and
$S_n(z)=z-\zeta_n$ if $\zeta_n\in I$ is the zero of $T_{n,x_1}$ in $I$. Then $Q_n$ has no zeros in $I$. Since $I$ is separated from both $\E$ and
$x_1$, we have, for some constant $C_1\ge1$ and all $n\ge1$,
\[
    \|Q_n\|_\E\le C_1e^{-nG_\E(x_1)}, \qquad
    C_1^{-1}\le |Q_n(x_1)|\le C_1.
\]

Since all zeros are real, the absence of zeros of $Q_n$ on $I$ gives a disk centered at $x_2$, independent of $n$, on which $Q_n$ has no zeros. Similarly, because $Q_n$ has no zeros in a neighborhood of $x_1$ contained in the gap containing $x_1$, there is such a disk centered at $x_1$. Connecting $x_1$ to $x_2$ by a path through the upper half-plane and taking a sufficiently small neighborhood of the path, we obtain a fixed connected open set $U\subset\Omega_\E\bs\{\infty\}$ containing $x_1$ and $x_2$ on which no $Q_n$ has a zero.

By the Bernstein--Walsh inequality and $\deg Q_n\le n$,
\[
    |Q_n(z)|
    \le \|Q_n\|_\E e^{nG_\E(z)}
    \le C_1e^{-nG_\E(x_1)+nG_\E(z)}, \quad z\in\Omega_\E.
\]
Hence
\[
    u_n(z)=\log C_1-nG_\E(x_1)+nG_\E(z)-\log|Q_n(z)|
\]
is nonnegative and harmonic on $U$. Since
$u_n(x_1)=\log C_1-\log|Q_n(x_1)|\le2\log C_1$, Harnack's inequality gives
$u_n(x_2)\le C_2$ for all $n$. Hence,
\[
    |Q_n(x_2)|\ge C_3e^{-nG_\E(x_1)+nG_\E(x_2)}.
\]
Therefore $Q_n/Q_n(x_2)$ is admissible for the residual problem at $x_2$, and we get
\[
    \|T_{n,x_2}\|_\E\le \frac{\|Q_n\|_\E}{|Q_n(x_2)|}\le C_4e^{-nG_\E(x_2)}.
\]
Multiplying by $e^{nG_\E(x_2)}+e^{-nG_\E(x_2)}$ shows that the Widom factors $\W_n(\E,x_2)$ are bounded. Thus boundedness for one finite normalization point implies boundedness for every other finite normalization point.

It remains to compare a finite normalization point with $\infty$.  Choose $x_*$ in the exterior gap of $\E$. Then it follows from Section~\ref{Sec2.3} that $d_n=n$, $\E_n$ is the same for both normalization points $x_*$ and $\infty$, and $\E_n$ does not intersect the exterior gap of $\E$, hence $\E_n\subset\cvh(\E)$.  By \eqref{WnDe-res},
\[
    \W_n(\E,x_*)
    =
    2\exp[n\De_{\E,n}(x_*)]
    \frac{1+\exp[-2nG_\E(x_*)]}{1+\exp[-2nG_{\E_n}(x_*)]}.
\]
The last factor is bounded above and below by $2$ and $2^{-1}$, respectively. Hence boundedness of $\W_n(\E,x_*)$ is equivalent to boundedness of $n\De_{\E,n}(x_*)$.

On the other hand, \eqref{WnDe-cheb} gives
\[
    \W_n(\E,\infty)=2\exp\left[n\De_{\E,n}(\infty)\right].
\]
Recall that $\De_{\E,n}(z)$ is nonnegative and harmonic in the domain $\Om_{\E_n}$ and hence also in the fixed domain $\ol{\bbC}\bs\cvh(\E)$.
By Harnack's inequality in this fixed domain, $\De_{\E,n}(x_*)$ and $\De_{\E,n}(\infty)$ are comparable by constants independent of $n$. Therefore boundedness of $\W_n(\E,x_*)$ is equivalent to boundedness of $\W_n(\E,\infty)$.
\end{proof}

Combining Theorem~\ref{Wn-ind-thm} with the converse theorem of \cite{CSYZ2} for generic regular sets gives the following residual version of that result.

\begin{corollary}\label{Wn-PW-eqv-cor}
Let $\E\subset\bbR$ be a regular compact set and suppose that $\E$ satisfies the generic rational-independence hypothesis of \cite{CSYZ2}. Then for every $x_*\in\ol{\bbR}\bs\E$,
\[
    \sup_{n\ge1}\W_n(\E,x_*)<\infty
    \quad\Longleftrightarrow\quad
    \PW(\E,x_*)<\infty.
\]
\end{corollary}

\begin{proof}
If $\PW(\E,x_*)<\infty$, boundedness follows from the regular Totik--Widom bound \eqref{UB-reg}. Conversely, if $\W_n(\E,x_*)$ is bounded for some $x_*$, Theorem~\ref{Wn-ind-thm} implies that $\W_n(\E,\infty)$ is bounded. The generic converse theorem of \cite{CSYZ2} gives $\PW(\E,\infty)<\infty$. By Corollary~\ref{IR-PW-inv}, finiteness of the Parreau--Widom constant is independent of the pole, so $\PW(\E,x_*)<\infty$.
\end{proof}

\section{Totik--Widom upper bound}\lb{Sec3}

In this section we prove upper bounds for the Widom factors.  Let $\E\subset\bbR$ be a compact non-polar set, $x_*\in\ol\bbR\bs\E$, and denote by $\cc G(\E,x_*)$ the collection of gaps of $\E$ which do not contain $x_*$.  Define the gap-height constant
\begin{align}\lb{g-height}
    \cc H(\E,x_*)=
    \sum_{K\in\cc G(\E,x_*)}\sup_{x\in K}G_\E(x,x_*).
\end{align}
The sum is interpreted as $\infty$ if it diverges and as zero if $\cc G(\E,x_*)$ is empty.

\begin{remark}
For $x_*=\infty$, the gap-height sum has a useful interpretation in terms of comb domains. Let $F=\overline{\E^\reg}$ which is the support of the equilibrium measure $\mu_\E$. Since $\E\bs F\subset\E^\ir$ is polar, the compact sets $\E$ and $F$ have the same Green function, in the sense that $G_F$ is the harmonic extension of $G_\E$ across the polar set $\E\bs F$. The comb construction of \cite{EY11} depends only on this Green function and recovers the closed support of the corresponding equilibrium measure, namely $F$. If $\E$ is semi-regular, equivalently, if $F$ is regular, the gap heights of $F$ are precisely the slit heights in the associated comb domain. However we note that the gap-height sum for $\E$ may be larger, since polar points of $\E\bs F$ can split a single gap of $F$ into several gaps of $\E$ without creating additional slits.
\end{remark}

\begin{theorem}\lb{UB1-thm}
Let $\E\subset\bbR$ be a compact set of positive capacity and $x_*\in\ol\bbR\bs\E$.  Then, for all $n\ge1$,
\begin{align}\lb{UB1}
    \W_n(\E,x_*)\le 2\exp\bigl[\cc H(\E,x_*)\bigr].
\end{align}
\end{theorem}

\begin{proof}
If $d_n=0$, then necessarily $n=1$, $T_{1,x_*}\equiv1$, and $x_*$ is not in the exterior gap.  Hence the exterior gap $K_\infty$ belongs to $\cc G(\E,x_*)$, and
\[
    \W_1(\E,x_*)=e^{G_\E(x_*)}+e^{-G_\E(x_*)}
    \le 2e^{G_\E(x_*)}
    =2e^{G_\E(\infty,x_*)}
    \le 2\exp\bigl[\cc H(\E,x_*)\bigr].
\]
Next assume $d_n\ge1$. Using the second equality in \eqref{De-rep} and \eqref{gap-meas}, we get
\begin{align}\lb{De-gap-est}
    \De_{\E,n}(x_*)
    =
    \int_{\E_n\bs\E}G_\E(x,x_*)\,d\mu_n(x)
    \le
    \frac1{d_n}\sum_{\substack{K\in\cc G(\E,x_*)\\ K\cap\E_n\ne\emptyset}}
    \sup_{x\in K}G_\E(x,x_*).
\end{align}

If $x_*=\infty$, then $d_n=n$, and so \eqref{UB1} follows from \eqref{WnDe-cheb} and \eqref{De-gap-est}.

It remains to consider the case $x_*\in\bbR\bs\E$. Since $\E\subset\E_n$, monotonicity of the Green function with respect to the underlying set gives $G_\E(x_*)\ge G_{\E_n}(x_*)$. It then follows from \eqref{WnDe-res} that
\[
    \W_n(\E,x_*)
    \le
    2\exp\big[d_n\De_{\E,n}(x_*)+(n-d_n)G_\E(x_*)\big].
\]
If $d_n=n$, then this is
\[
    \W_n(\E,x_*)\le 2\exp[n\De_{\E,n}(x_*)],
\]
and hence \eqref{UB1} follows from \eqref{De-gap-est}. If $d_n=n-1$, then
\[
    \W_n(\E,x_*)\le
    2\exp[d_n\De_{\E,n}(x_*)+G_\E(x_*)].
\]
In this case, by the properties of $\E_n$ recalled in Section~\ref{Sec2.3}, the exterior gap $K_\infty$ of $\E$ does not meet $\E_n$. Since $\infty\in K_\infty$, we have
\[
    G_\E(x_*) = G_\E(\infty,x_*) \le \sup_{x\in K_\infty}G_\E(x,x_*).
\]
Combining these estimates with \eqref{De-gap-est} gives
\[
    \log\frac{\W_n(\E,x_*)}{2}
    \le
    \sum_{K\in\cc G(\E,x_*)}\sup_{x\in K}G_\E(x,x_*),
\]
which proves \eqref{UB1}.
\end{proof}

Next we show that outside the semi-regular class the height constant is always infinite.

\begin{theorem}
Let $\E\subset\bbR$ be a compact non-polar set. If $\E$ is not semi-regular, then
$\cc H(\E,x_*)=\infty$ for every $x_*\in\ol\bbR\bs\E$.
\end{theorem}

\begin{proof}
We prove the contrapositive. Suppose that $\cc H(\E,x_*)<\infty$ for some $x_*\in\ol\bbR\bs\E$. We will show that $\E^\reg$ is closed.
If $\E$ has only finitely many gaps, then it is a union of finitely many closed possibly degenerate intervals. In this case $\E^\reg$ is the finite union of nondegenerate closed intervals, and hence is automatically closed. Thus, it suffices to assume that $\E$ has infinitely many gaps.

Let $r\in\ol{\E^\reg}$. Since $\E$ is compact, $r\in\E$.
Let $\cc G_b(\E,x_*)$ denote the bounded gaps in $\cc G(\E,x_*)$. Fix $\eps>0$. By the finiteness assumption $\cc H(\E,x_*)<\infty$, choose a finite collection $\cc F\subset\cc G_b(\E,x_*)$ such that
\begin{align}\label{ccF-eps}
    \sum_{K\in\cc G_b(\E,x_*)\bs\cc F}\sup_{x\in K}G_\E(x,x_*)<\eps.
\end{align}
Fill all bounded gaps except those in $\cc F$ and except the gap containing $x_*$,
\begin{align}\label{F-eps}
    F_\eps=\E\cup\bigcup_{K\in\cc G_b(\E,x_*)\bs\cc F}K.
\end{align}
Then $F_\eps$ is compact, $x_*\notin F_\eps$, and $F_\eps$ has only finitely many bounded gaps.
Since $F_\eps\supset\E$, monotonicity gives $G_{F_\eps}(z,x_*)\le G_\E(z,x_*)$. The difference
\[
u=G_\E(\cdot,x_*)-G_{F_\eps}(\cdot,x_*)
\]
has no pole, because the singularities cancel, and is harmonic and bounded above on $\ol\bbC\bs F_\eps$. 
On $\E$, its boundary values are zero quasi-everywhere. On each filled gap $K\in\cc G_b(\E,x_*)\bs\cc F$, we have $G_{F_\eps}=0$ quasi-everywhere by \eqref{F-eps} and $G_\E(\cdot,x_*)<\eps$ by \eqref{ccF-eps}.
Hence the boundary values of $u$ on $F_\eps$ are bounded above by $\eps$ quasi-everywhere. Then, by the maximum principle, for all $z$,
\begin{align}\lb{GF-comp}
    0\le G_\E(z,x_*)-G_{F_\eps}(z,x_*)\le\eps.
\end{align}

Since $r\in\ol{\E^\reg}$ and $\E^\reg\subset F_\eps$, the point $r$ is not isolated in $F_\eps$. Since $F_\eps$ has only finitely many bounded gaps, it is a finite union of closed intervals, allowing degenerate intervals. Hence $r$ belongs to a non-degenerate interval $I\subset F_\eps$. By monotonicity, $0\le G_{F_\eps}(z,x_*)\le G_I(z,x_*)$. The Green function of an interval is continuous up to the interval and vanishes there, so $G_{F_\eps}(z,x_*)\to0$ as $z\to r$.

From \eqref{GF-comp}, $\limsup_{z\to r}G_\E(z,x_*)\le\eps$. Since $\eps$ was arbitrary, the upper boundary value of $G_\E(\cdot,x_*)$ at $r$ is zero. Thus $r\in\E^\reg$. Therefore $\ol{\E^\reg}\subset\E^\reg$, so $\E$ is semi-regular. This proves the contrapositive.
\end{proof}

For semi-regular sets, the gap heights in Theorem~\ref{UB1-thm} can be estimated in terms of the Parreau--Widom constant and the irregularity coefficient.

\begin{theorem}\lb{UB2-thm}
Let $\E\subset\bbR$ be a compact semi-regular set and let $x_*\in\ol\bbR\bs\E$.  Then $\cc H(\E,x_*)\le \PW(\E,x_*)+\IR(\E,x_*)$, and hence, for all $n\ge1$,
\begin{align}\lb{UB2}
    \W_n(\E,x_*)\le 2\exp\big[\PW(\E,x_*)+\IR(\E,x_*)\big].
\end{align}
\end{theorem}

\begin{proof}
Since $\E$ is semi-regular, $G_\E(\cdot,x_*)=G_{\E^\reg}(\cdot,x_*)$ and this function is continuous away from the pole. Let $K$ be a gap of $\E$ which does not contain $x_*$, and let $L$ be the gap of $\E^\reg$ containing $K$.

If $L$ does not contain $x_*$, then $G_{\E^\reg}(\cdot,x_*)$ has a unique critical point in $L$, where it attains its maximum over $L$, and is monotone on the two sides of this point. Thus the supremum over $K$ is either the corresponding critical value, if the critical point lies in $\ol K$, or else it is the boundary value $G_\E(p,x_*)$ at an endpoint $p$ of $K$. In the latter case $p\in\E^\ir$, since the endpoint selected by monotonicity lies in $L$, while $L\cap\E^\reg=\emptyset$.

If $L$ contains $x_*$, then there is no critical point in $L\bs\{x_*\}$, and the same monotonicity on each component of $L\bs\{x_*\}$ shows that the supremum over every counted gap $K\subset L$ is again attained at an irregular endpoint.

Inside a fixed gap $L$ of $\E^\reg$, monotonicity on the two sides of the critical point shows that an irregular endpoint can be selected by at most one adjacent gap, unless the endpoint is itself the critical point. In that exceptional case the two adjacent gaps may both contribute the same value, but one copy is accounted for by the irregularity sum and the other by the Parreau--Widom critical value. If $L$ contains $x_*$, the gap containing $x_*$ is not counted and there is no critical point, so each irregular endpoint is selected at most once.

Summing over all gaps of $\E$ which do not contain $x_*$, the critical-value contributions are bounded by $\PW(\E,x_*)$ and the irregular-endpoint contributions are bounded by $\IR(\E,x_*)$. Hence $\cc H(\E,x_*)\le \PW(\E,x_*)+\IR(\E,x_*)$, and \eqref{UB2} follows from Theorem~\ref{UB1-thm}.
\end{proof}

\section{Asymptotic lower bound}\lb{Sec4}

In this section we prove lower bounds. The first one holds for arbitrary compact non-polar sets and involves only the irregular points separated from the regular part. For semi-regular sets this gives a lower bound complementary to \eqref{UB2}.

We first recall a standard root-asymptotic fact.

\begin{lemma}\label{root-lem}
Let $\E$ be a compact non-polar set and $x_*\in\ol\bbR\bs\E$. For each $n$, let $P_n$ be a polynomial of degree $m_n\le n$ with $m_n/n\to1$. Assume that the polynomials are normalized at $x_*$ in the same way as the extremal polynomials, that is, if $x_*=\infty$, then $P_n$ is monic of degree $m_n$, while if $x_*\in\bbR\bs\E$, then $P_n(x_*)=1$. Suppose that
\begin{align}\label{asymp-extr}
    \lim_{n\to\infty}\|P_n\|_\E^{1/n}=\ca(\E,x_*)
\end{align}
and suppose further that there exists a connected open set $U\subset\Omega_\E$ containing $x_*$ such that no $P_n$ has a zero in $U$. Then, locally uniformly on $U\bs\{\infty\}$,
\begin{align}\label{root-asymp}
    |P_n(z)|^{1/n}\to \ca(\E,x_*)e^{G_\E(z)}.
\end{align}
\end{lemma}

\begin{proof}
Define
\[
    h_n(z)=
    \log \|P_n\|_\E^{1/n}
    +\frac{m_n}{n}G_\E(z)
    -\frac1n\log|P_n(z)|.
\]
By Bernstein--Walsh, $h_n\ge0$ on $U$, and since $P_n$ has no zeros on $U$, the functions $h_n$ are harmonic there.

If $x_*=\infty$, then the singularity at infinity is removable and
\[
    h_n(\infty)
    =
    \log \|P_n\|_\E^{1/n}
    -\frac{m_n}{n}\log\ca(\E)
    \to0.
\]
If $x_*\in\bbR\bs\E$, then $P_n(x_*)=1$ and $\log\ca(\E,x_*)=-G_\E(x_*)$, so
\[
    h_n(x_*)
    =
    \log\frac{\|P_n\|_\E^{1/n}}{\ca(\E,x_*)}
    +\left(\frac{m_n}{n}-1\right)G_\E(x_*)
    \to0.
\]
Harnack's inequality gives $h_n\to0$ locally uniformly on $U$.  Therefore
\[
    \frac1n\log|P_n(z)|
    =
    \log \|P_n\|_\E^{1/n}
    +\frac{m_n}{n}G_\E(z)-h_n(z)
    \to
    \log\ca(\E,x_*)+G_\E(z),
\]
which proves \eqref{root-asymp}.
\end{proof}

Next we prove a technical lemma about the sets $\E_n$ introduced in Section~\ref{Sec2.3}.

\begin{lemma}\label{shrink-lem}
Let $\E\subset\bbR$ be a compact non-polar set and $x_*\in\ol\bbR\bs\E$. Suppose $r$ is an isolated point of $\E$.
For each $n\ge2$, let $\E_{n,r}$ denote the component of $\E_n$ containing $r$. Then
\begin{align}\label{shrink}
    \lim_{n\to\infty}\diam(\E_{n,r})=0.
\end{align}
\end{lemma}

\begin{proof}
Choose an interval $J=[r-\delta,r+\delta]$, $\delta>0$, such that $J\cap\E=\{r\}$ and $x_*\notin J$. Let $\E'=\E\bs\{r\}$ and note that it is compact since $r$ is an isolated point of $\E$. Removing the polar point $r$ does not change the Green function or the capacity, so $G_\E=G_{\E'}$ on $\Omega_\E$ and $\ca(\E,x_*)=\ca(\E',x_*)$. Since $J\subset\Omega_{\E'}$, the function $G_{\E'}$ is continuous and positive on $J$. Pick $\eta>0$ so that $G_{\E'}\ge2\eta$ on $J$.

Let $R_n$ be the monic polynomial whose zeros are precisely the zeros of $T_{n,x_*}$ in $J$ and write $T_{n,x_*}=R_nQ_n$. By the band description in Section~\ref{Sec2.3}, each band of $\E_n$ contains one simple zero of $T_{n,x_*}$ and each gap of $\E$ meets at most one band. Since the bands are nondegenerate intervals and $J\cap\E=\{r\}$, every band meeting $J$ meets one of the two gaps adjacent to $r$. Thus at most two bands meet $J$, and hence $\deg R_n\le2$.

Set
\[
    P_n=
    \begin{cases}
        Q_n, & x_*=\infty,\\[1mm]
        Q_n/Q_n(x_*), & x_*\in\bbR\bs\E,
    \end{cases}
    \qquad
    m_n=\deg P_n.
\]
Then $n-3\le m_n\le n$. Since the zeros of $R_n$ lie in $J$ and
$\E'$ is disjoint from $J$, the factors in $R_n$ are bounded above
and below in modulus on $\E'$ by constants independent of $n$.
The same is true at $x_*$ when $x_*$ is finite. Thus, for some
constant $C>0$ independent of $n$,
\[
    C^{-1}\|T_{n,x_*}\|_\E
    =
    C^{-1}\|T_{n,x_*}\|_{\E'}
    \le
    \|P_n\|_{\E'}
    \le
    C\|T_{n,x_*}\|_\E.
\]
The first equality holds since by the alternation theorem $T_{n,x_*}$ has at least two extremal points on $\E$, that is, points where $|T_{n,x_*}|=\|T_{n,x_*}\|_\E$.
Then, the root asymptotics, $\|T_{n,x_*}\|_{\E}^{1/n}\to\ca(\E,x_*)=\ca(\E',x_*)$, implies $\|P_n\|_{\E'}^{1/n}\to\ca(\E',x_*)$.

Choose a closed interval $I\subset(r-\delta,r+\delta)$ with $r$
in its interior. All zeros of $P_n$ are real and lie outside $J$,
and $P_n$ has no zeros in the gap of $\E$ containing $x_*$. Hence
there is a fixed connected open set $U\subset\Omega_{\E'}$
containing $I$ and $x_*$ such that no $P_n$ has a zero in $U$.
Indeed, one can take the upper half-plane together with sufficiently
small neighborhoods of $I$ and of $x_*$, where in the Chebyshev
case the latter is a neighborhood of infinity.

By Lemma~\ref{root-lem}, uniformly for $x\in I$,
\[
    \frac1n\log|P_n(x)|
    \to
    \log\ca(\E',x_*)+G_{\E'}(x).
\]
In the Chebyshev case $Q_n=P_n$. In the residual case,
$Q_n(x_*)=1/R_n(x_*)$, so $|Q_n(x_*)|$ is bounded above and below
by positive constants independent of $n$. Combining the preceding
limit with the root asymptotics for $\|T_{n,x_*}\|_\E$ gives,
uniformly for $x\in I$,
\[
    \frac1n\log
    \frac{|Q_n(x)|}{\|T_{n,x_*}\|_\E}
    \to
    G_{\E'}(x).
\]
Since $G_{\E'}\ge2\eta$ on $I$, it follows that, for all sufficiently
large $n$,
$|Q_n(x)|\ge e^{\eta n}\|T_{n,x_*}\|_\E$ for every $x\in I$.

If $x\in\E_n\cap I$, then $|T_{n,x_*}(x)|\le\|T_{n,x_*}\|_\E$, and therefore $|R_n(x)|\le e^{-\eta n}$.
Since $r\in\E_n\cap I$, this also shows that $R_n$ is nonconstant for all sufficiently large $n$.
As $R_n$ is monic and $\deg R_n\le2$, every $x\in\E_n\cap I$ lies within $\rho_n=e^{-\eta n/2}$ of a zero of $R_n$.

Let $C_n$ be the component of $\E_n\cap I$ containing $r$. Then $C_n$ is contained in a connected component of the union of at most
two intervals of radius $\rho_n$ centered at the zeros of $R_n$. Every such connected component has diameter at most $4\rho_n$, and
hence
\[
    \diam(C_n)\le4e^{-\eta n/2}\to0.
\]
Since $r$ lies in the interior of $I$, the component $C_n$ does not meet the endpoints of $I$ for all sufficiently large $n$. Therefore
$C_n$ is the full component $\E_{n,r}$ of $\E_n$ containing $r$, and \eqref{shrink} follows.
\end{proof}

We now use the lemma to prove an asymptotic lower bound for arbitrary compact non-polar sets in terms of the following restricted irregularity coefficient.

\begin{align}\label{IR0-def}
    \IR_0(\E,x_*)=\sup_F\sum_{x\in F}G_\E(x,x_*),
\end{align}
where the supremum is over all finite sets $F\subset\E\bs\ol{\E^\reg}$. Thus $\IR_0(\E,x_*)\le\IR(\E,x_*)$. If $\E$ is semi-regular then $\ol{\E^\reg}=\E^\reg$ and hence $\IR_0(\E,x_*)=\IR(\E,x_*)$.

\begin{theorem}\label{ALB0-thm}
Let $\E\subset\bbR$ be a compact non-polar set and let $x_*\in\ol\bbR\bs\E$. Then
\begin{align}\label{ALB0}
    \liminf_{n\to\infty}\W_n(\E,x_*)\ge 2\exp\big[\IR_0(\E,x_*)\big].
\end{align}
\end{theorem}

\begin{proof}
If $\E=\ol{\E^\reg}$, then $\IR_0(\E,x_*)=0$ and \eqref{ALB0} follows from \eqref{LB}. Now assume that $\E\bs\ol{\E^\reg}$ is nonempty. Let $F=\{x_1,\ldots,x_m\}\subset\E\bs\ol{\E^\reg}$ and set $K=\ol{\E^\reg}\cup F$.
Then $K$ is compact and the points of $F$ are isolated in $K$.
Since $\E\bs K\subset\E^\ir$ is polar, the Green functions and the capacities for $\E$ and $K$ agree.

Let $K_n$ be the inverse-image set associated with $K$, and let $d_n$ and $\mu_n$ denote the corresponding degree and
equilibrium measure. By Lemma~\ref{shrink-lem}, the components of $K_n$ containing $x_1,\ldots,x_m$ shrink to these points and are pairwise disjoint for all sufficiently large $n$. The points $x_\ell$ are irregular in $K$, and $G_K(\cdot,x_*)$ is continuous at each of them. Thus, given $\eps>0$, these components eventually lie in pairwise disjoint neighborhoods that do not intersect $K^\reg$ and on which $G_K(\cdot,x_*)\ge G_K(x_\ell,x_*)-\eps/m$. Each component has $\mu_n$-measure at least $1/d_n$, so \eqref{De-rep} gives
\[
    d_n\De_{K,n}(x_*)
    \ge
    \sum_{\ell=1}^mG_K(x_\ell,x_*)-\eps
\]
for all sufficiently large $n$. This yields the desired lower bound \eqref{ALB0} for the set $K$. Indeed, in the Chebyshev case it follows from \eqref{WnDe-cheb} and $d_n=n$, and in the residual case from \eqref{WnDe-res-asymp} and $(n-d_n)G_K(x_*)\ge0$. Letting $\eps\downarrow0$, we obtain
\[
    \liminf_{n\to\infty}\W_n(K,x_*)
    \ge
    2\exp\bigg[\sum_{x\in F}G_K(x,x_*)\bigg].
\]
Since $K\subset\E$ and their capacities agree, monotonicity of the extremal norm gives $\W_n(\E,x_*)\ge\W_n(K,x_*)$. Also $G_K=G_\E$, so
\[
    \liminf_{n\to\infty}\W_n(\E,x_*)
    \ge
    2\exp\bigg[\sum_{x\in F}G_\E(x,x_*)\bigg].
\]
Taking the supremum over all finite sets $F\subset\E\bs\ol{\E^\reg}$ proves \eqref{ALB0}.
\end{proof}

The lower bound \eqref{ALB0} implies $\W_n(\E,x_*)\to\infty$ whenever $\IR_0(\E,x_*)=\infty$. The following corollary gives a simple condition for this in terms of accumulation of irregular points away from the regular part.

\begin{corollary}\label{accum-cor}
Let $\E\subset\bbR$ be a compact non-polar set and let $x_*\in\ol\bbR\bs\E$. If $\E^\ir$ has an accumulation point in $\E\bs\ol{\E^\reg}$, then $\W_n(\E,x_*)\to\infty$. In particular, this holds if $\E^\ir\bs\ol{\E^\reg}$ is uncountable.
\end{corollary}

\begin{proof}
Let $r\in\E\bs\ol{\E^\reg}$ be an accumulation point of $\E^\ir$ and choose distinct points $x_j\in\E^\ir$ with $x_j\to r$. Since $r\notin\ol{\E^\reg}$, all but finitely many $x_j$ lie in $\E\bs\ol{\E^\reg}$. Also $\E\bs\ol{\E^\reg}$ is polar, so $G_\E=G_{\ol{\E^\reg}}$ and this Green function is harmonic in a neighborhood of $r$. Hence $G_\E(x_j,x_*)\to G_\E(r,x_*)>0$ as $j\to\infty$. Thus $\IR_0(\E,x_*)=\infty$, and Theorem~\ref{ALB0-thm} gives $\liminf_{n\to\infty}\W_n(\E,x_*)=\infty$.

If $\E^\ir\bs\ol{\E^\reg}$ is uncountable, then it follows from
\[
    \E^\ir\bs\ol{\E^\reg}\subset\bigcup_{k\ge1}\{x\in\E:\dist(x,\ol{\E^\reg})\ge1/k\}
\]
that one of the compact sets on the right contains uncountably many points of $\E^\ir$. This uncountable subset has an accumulation point in the same compact set, and this accumulation point lies in $\E\bs\ol{\E^\reg}$. Now the first part of the proof applies.
\end{proof}

We now return to the compact semi-regular setting. In this case
$\ol{\E^\reg}=\E^\reg$, so the restricted coefficient $\IR_0$ is the full
irregularity coefficient $\IR$. Thus Theorem~\ref{ALB0-thm} and
Corollary~\ref{accum-cor} give the following.

\begin{corollary}\label{ALB-cor}
Let $\E\subset\bbR$ be a compact semi-regular set and let
$x_*\in\ol\bbR\bs\E$. Then
\begin{align}\label{ALB}
    \liminf_{n\to\infty}\W_n(\E,x_*)\ge 2\exp\big[\IR(\E,x_*)\big].
\end{align}
Consequently, if $\IR(\E,x_*)=\infty$, then $\W_n(\E,x_*)\to\infty$.
In particular, $\W_n(\E,x_*)\to\infty$ if $\E^\ir$ has an accumulation point in
$\E^\ir$ or if $\E^\ir$ is uncountable.
\end{corollary}

\section{Widom factors of semi-regular Parreau--Widom sets}\lb{Sec5}

We now combine the upper and lower bounds to characterize bounded Widom factors in the semi-regular Parreau--Widom class.

\begin{theorem}\label{BddEqv-thm}
Let $\E\subset\bbR$ be a semi-regular Parreau--Widom set and $x_*\in\ol\bbR\bs\E$.  Then the Widom factors $\W_n(\E,x_*)$ are bounded if and only if $\IR(\E,x_*)<\infty$.
\end{theorem}

\begin{proof}
If $\IR(\E,x_*)=\infty$, Corollary~\ref{ALB-cor} gives $\liminf_{n\to\infty}\W_n(\E,x_*)=\infty$, so the Widom factors are unbounded.  Conversely, suppose $\IR(\E,x_*)<\infty$. By the Parreau--Widom assumption, $\PW(\E,x_*)<\infty$ and hence $\W_n(\E,x_*)$ are bounded by \eqref{UB2}.
\end{proof}

The next result translates the finiteness of $\IR$ into a geometric condition when the regular part has finitely many gaps.

\begin{theorem}\label{BlaCond-thm}
Assume that $\E\subset\bbR$ is semi-regular, $x_*\in\ol\bbR\bs\E$, and that $\E^\reg$ is a finite union of closed intervals. Then
\begin{align}\label{Bla-cond}
    \IR(\E,x_*)<\infty
    \quad\Longleftrightarrow\quad
    \sum_{x\in\E^\ir}\dist(x,\E^\reg)^{1/2}<\infty,
\end{align}
where the sum on the right is understood to be infinite if $\E^\ir$ is uncountable.
\end{theorem}

\begin{proof}
It suffices to consider the case where $\E^\ir$ is countable. Since $\E^\reg$ is a finite union of intervals, Proposition~4.8 of \cite{CSZ10} implies that, for any sequence $\{x_j\}\subset\bbR\bs\E^\reg$,
\[
    \sum_j G_{\E^\reg}(x_j)<\infty
    \quad\Longleftrightarrow\quad
    \sum_j \dist(x_j,\E^\reg)^{1/2}<\infty.
\]
Since $\E$ is semi-regular, $G_\E(\cdot,\infty)=G_{\E^\reg}(\cdot,\infty)$ on $\bbR\bs\E^\reg$. Applying this to the irregular points gives
\[
    \IR(\E,\infty)<\infty
    \quad\Longleftrightarrow\quad
    \sum_{x\in\E^\ir}\dist(x,\E^\reg)^{1/2}<\infty.
\]
Finally, by Corollary~\ref{IR-PW-inv}, finiteness of $\IR(\E,x_*)$ is independent of the pole $x_*\in\ol\bbR\bs\E$. This proves \eqref{Bla-cond}.
\end{proof}

\begin{corollary}\lb{BddEqv-cor}
Suppose $\E\subset\bbR$ is semi-regular and $\E^\reg$ is a finite union of closed intervals.  Then, for every $x_*\in\ol\bbR\bs\E$, the Widom factors $\W_n(\E,x_*)$ are bounded if and only if
\begin{align*}
    \sum_{x\in\E^\ir}\dist(x,\E^\reg)^{1/2}<\infty.
\end{align*}
\end{corollary}

\begin{proof}
In the present setting $\E^\reg$ has finitely many gaps, so $\PW(\E,x_*)$ is a finite sum. Thus, the result follows from Theorems~\ref{BddEqv-thm} and~\ref{BlaCond-thm}.
\end{proof}

When the regular part is a single interval, the upper and lower bounds coincide.

\begin{theorem}\label{Wlim-int-thm}
Suppose that $\E=[a,b]\cup\{x_k\}_{k\ge1}$ is compact, $x_k\in\bbR\bs[a,b]$ are distinct, and  $x_*\in\ol\bbR\bs\E$.  Then the Widom factors are bounded if and only if
\begin{align}\label{int-sqrt-cond}
    \sum_{k=1}^{\infty}\dist(x_k,[a,b])^{1/2}<\infty.
\end{align}
Moreover, the Widom factors admit the asymptotic formula
\begin{align}\label{int-lim}
    \lim_{n\to\infty}\W_n(\E,x_*)=2\exp\big[\IR(\E,x_*)\big].
\end{align}
\end{theorem}

\begin{proof}
Since the countable set $\{x_k\}$ is polar, $\E^\reg=[a,b]$ and $G_\E=G_{[a,b]}$.  The Parreau--Widom contribution vanishes, $\PW(\E,x_*)=0$.  The upper bound \eqref{UB2} and the lower bound \eqref{ALB} give
\begin{align*}
    2\exp\big[\IR(\E,x_*)\big]
    \le
    \liminf_{n\to\infty}\W_n(\E,x_*)
    \le
    \limsup_{n\to\infty}\W_n(\E,x_*)
    \le
    2\exp\big[\IR(\E,x_*)\big],
\end{align*}
which proves \eqref{int-lim}.  The boundedness criterion follows from Corollary~\ref{BddEqv-cor}.
\end{proof}

\begin{example}
Let $\alpha>0$ and $\E_\alpha=[-1,1]\cup\{1+k^{-\alpha}\}_{k\ge1}$.  Then, for every $x_*\in\ol\bbR\bs\E_\alpha$, the Widom factors $\W_n(\E_\alpha,x_*)$ are bounded if and only if $\alpha>2$.
\end{example}

We finish with Szeg\H{o}--Widom asymptotics for the normalized extremal polynomials in the interval-with-points case.  For the remainder of the section we assume that $\E=[a,b]\cup\{x_k\}_{k\ge1}$, the points $x_k\in\bbR\bs[a,b]$ are distinct, and \eqref{int-sqrt-cond} holds.  Set $\Omega=\ol\bbC\bs[a,b]$ and define
\begin{align*}
    B(z)=
    \frac{z-\frac{a+b}{2}-\sqrt{(z-a)(z-b)}}{\frac{b-a}{2}},
\end{align*}
where the branch is chosen so that $\sqrt{(z-a)(z-b)}=z-(a+b)/2+O(z^{-1})$ as $z\to\infty$.  Then $B$ maps $\Omega$ conformally onto $\bbD$, $B(\infty)=0$, $|B(z)|=\exp[-G_{[a,b]}(z)]$, and
\begin{align}\label{B-infty}
    B(z)=\frac{\ca([a,b])}{z}+O(z^{-2}),\quad z\to\infty.
\end{align}

For $y\in\bbR\bs[a,b]$, set $\beta_y=B(y)$ and define
\begin{align*}
    b_y(z)=\frac{|\beta_y|}{\beta_y}\frac{\beta_y-B(z)}{1-\ol{\beta_y}B(z)}.
\end{align*}
Then $b_y$ has a simple zero at $y$ and unimodular boundary values on $[a,b]$. It follows that $|b_y(z)|=\exp[-G_{[a,b]}(z,y)]$.  The Blaschke condition \eqref{int-sqrt-cond} is equivalent to locally uniform convergence in $\Omega$ of the product
\begin{align*}
    A(z)=\prod_{k=1}^{\infty}b_{x_k}(z).
\end{align*}
Consequently, for $z\in\Om\bs\{x_k\}_{k\ge1}$,
\begin{align}\label{A-G}
    |A(z)|
    =
    \exp\Bigl[-\sum_{k=1}^{\infty}G_{[a,b]}(z,x_k)\Bigr],
\end{align}
with both sides extending by zero at the points $x_k$. In particular, at $z=x_*$ we have
\begin{align}\label{A-IR}
    |A(x_*)|=\exp[-\IR(\E,x_*)].
\end{align}

For $x_*\in\ol{\bbR}\bs\E$, define
\[
    B_{x_*}(z)=
    \begin{cases}
    B(z), & x_*=\infty,\\[1mm]
    \dfrac{|B(x_*)|}{B(x_*)}B(z), & x_*\in\bbR\bs\E,
    \end{cases}
    \quad
    A_{x_*}(z)=
    \begin{cases}
    A(z), & x_*=\infty,\\[1mm]
    \dfrac{|A(x_*)|}{A(x_*)}A(z), & x_*\in\bbR\bs\E.
    \end{cases}
\]
Then for $x_*\in\bbR\bs\E$,
\[
    B_{x_*}(x_*)=|B(x_*)|>0
    \quad\text{and}\quad
    A_{x_*}(x_*)=|A(x_*)|>0,
\]
and at infinity $\lim_{z\to\infty}zB(z)=\ca([a,b])>0$ and $A_\infty(\infty)=A(\infty)>0$.

\begin{lemma}\label{SW-bound-lem}
Assume that $\E=[a,b]\cup\{x_k\}_{k\ge1}$, the points
$x_k\in\bbR\bs[a,b]$ are distinct and satisfy \eqref{int-sqrt-cond}. Fix
$x_*\in\ol\bbR\bs\E$ and set
\[
    F_n(z)=\frac{T_{n,x_*}(z)}{\|T_{n,x_*}\|_\E}B_{x_*}(z)^n.
\]
If $F_{n_j}$ converges locally uniformly in $\Omega$ to $F$, then
$F(x_k)=0$ for every $k$ and
\[
    |F(z)|\le \frac12 |A_{x_*}(z)|,\quad z\in\Omega.
\]
\end{lemma}

\begin{proof}
To prove the result we work along the convergent subsequence and write $n$ instead of $n_j$. Since $\{x_k\}_{k\ge1}$ is polar, $G_\E=G_{[a,b]}$. Fix $z\in\Omega\bs\{\infty\}$. If $z\in\E_n$ for infinitely many $n$, in particular if $z=x_k$, then along that subsequence $|F_n(z)|\le|B_{x_*}(z)|^n= e^{-nG_\E(z)}\to0$, and hence $F(z)=0$. Thus, $F(x_k)=0$ for every $k$ and to prove the inequality it suffices to assume $z\in\Omega_{\E_n}\bs\{\infty\}$ for large $n$.

By \eqref{Tn-pointwise}, $|B_{x_*}(z)|=e^{-G_\E(z)}$, \eqref{De-def}, and $d_n\le n$, we have
\begin{align}\label{F-est}
    |F_n(z)|
    &\le
    \frac12 \bigl(e^{d_nG_{\E_n}(z)} + e^{-d_nG_{\E_n}(z)}\bigr)e^{-n G_\E(z)} \no \\
    &\le
    \frac12 e^{-d_n\De_{\E,n}(z)}
    +
    \frac12 e^{-nG_\E(z)}.
\end{align}
Since $G_\E(z)>0$, the last term tends to zero. To estimate the first term we proceed as in the proof of Theorem~\ref{ALB0-thm}, but with the pole of the Green function at $z$. The summability condition \eqref{int-sqrt-cond} implies that every $x_k$ is isolated in $\E$. Then, for every fixed $m$, Lemma~\ref{shrink-lem} shows that the components of $\E_n$ containing $x_1,\ldots,x_m$ are pairwise disjoint, shrink to these points, and
lie in $\E_n\bs\E^\reg$ for all sufficiently large $n$. Since each component has $\mu_n$-measure at least $1/d_n$ and $G_\E(\cdot,z)$ is continuous at $x_1,\ldots,x_m$, the first equality in \eqref{De-rep-z} gives
\[
    \liminf_{n\to\infty} d_n\De_{\E,n}(z)
    \ge
    \sum_{k=1}^m G_\E(x_k,z).
\]
Then letting $m\to\infty$ gives
\begin{align}\label{dnDe-est}
    \liminf_{n\to\infty} d_n\De_{\E,n}(z)
    \ge
    \sum_{k=1}^\infty G_\E(x_k,z).
\end{align}
Therefore by \eqref{F-est}, \eqref{dnDe-est}, and \eqref{A-G},
\[
    |F(z)|
    \le
    \frac12
    \exp\Bigl[-\sum_{k=1}^\infty G_\E(x_k,z)\Bigr]
    =
    \frac12|A_{x_*}(z)|, \quad z\in\Om\bs\{\infty\}.
\]
The inequality at infinity then follows by taking the limit $z\to\infty$.
\end{proof}

\begin{theorem}[Szeg\H{o}--Widom asymptotics]\label{Slim-int-thm}
Assume that $\E=[a,b]\cup\{x_k\}_{k\ge1}$, the points $x_k\in\bbR\bs[a,b]$ are distinct and satisfy \eqref{int-sqrt-cond}.
Then, for every $x_*\in\ol{\bbR}\bs\E$,
\begin{align*}
    \frac{T_{n,x_*}(z)}{\|T_{n,x_*}\|_\E}B_{x_*}(z)^n
    \to
    \frac12 A_{x_*}(z)
\end{align*}
locally uniformly on compact subsets of $\Omega$.
\end{theorem}

\begin{proof}
Define
\[
    F_n(z)=\frac{T_{n,x_*}(z)}{\|T_{n,x_*}\|_\E}B_{x_*}(z)^n.
\]
Since $G_\E=G_{[a,b]}$, $|B_{x_*}(z)|=e^{-G_\E(z)}$, and
$\deg T_{n,x_*}\le n$, Bernstein--Walsh inequality \eqref{BW-ineq} gives
local boundedness of $\{F_n\}$ in $\Omega$. Hence $\{F_n\}$ is a normal
family.

Let $F$ be a subsequential limit. By Lemma~\ref{SW-bound-lem},
\[
    |F(z)|\le \frac12|A_{x_*}(z)|,\qquad z\in\Omega,
\]
and $F(x_k)=0$ for every $k$. Hence
\[
    H(z)=\frac{2F(z)}{A_{x_*}(z)}
\]
has removable singularities at the points $x_k$, is analytic in $\Omega$,
and satisfies $|H|\le1$.

We identify the normalization value. If $x_*=\infty$, then by \eqref{B-infty},
the monicity of $T_n$, \eqref{int-lim}, and \eqref{A-IR},
\[
    F(\infty)=
    \lim_{n\to\infty}\frac{\ca([a,b])^n}{\|T_n\|_\E}
    =
    \frac12A_\infty(\infty).
\]
If $x_*\in\bbR\bs\E$, then $T_{n,x_*}(x_*)=1$ and
$B_{x_*}(x_*)=|B(x_*)|$. Using \eqref{Wn-res}, \eqref{int-lim}, and
\eqref{A-IR}, we obtain
\[
    F(x_*)=
    \lim_{n\to\infty}\frac{|B(x_*)|^n}{\|T_{n,x_*}\|_\E}
    =
    \lim_{n\to\infty}\frac{1+|B(x_*)|^{2n}}{\W_n(\E,x_*)}
    =
    \frac12A_{x_*}(x_*).
\]
Thus $H(x_*)=1$ in each case. By the
maximum principle, $H\equiv1$. Hence every subsequential limit is
$A_{x_*}/2$, and the full sequence converges locally uniformly on compact
subsets of $\Omega$.
\end{proof}




\begin{thebibliography}{99}
















\bibitem{Chr12} J.\ S.\ Christiansen, \emph{Szeg\H{o}'s theorem on Parreau-Widom sets}, Adv.\ Math.\ \textbf{229} (2012), 1180--1204.

\bibitem{Chr19} J.\ S.\ Christiansen, \emph{Dynamics in the Szeg\H{o} class and polynomial asymptotics}, J.\ Anal.\ Math.\ \textbf{137} (2019), no.~2, 723--749.

\bibitem{CSZ10} J.\ S.\ Christiansen, B.\ Simon, and M.\ Zinchenko, \emph{Finite gap Jacobi matrices, I. The isospectral torus}, Constr.\ Approx.\ \textbf{32} (2010), no.~1, 1--65.

\bibitem{CSZ11} J.\ S.\ Christiansen, B.\ Simon, and M.\ Zinchenko, \emph{Finite gap Jacobi matrices, II.\ The Szeg\H{o} class}, Constr.\ Approx.\ \textbf{33} (2011), 365--403.

\bibitem{CSZ1} J.\ S.\ Christiansen, B.\ Simon, and M.\ Zinchenko, \emph{Asymptotics of Chebyshev Polynomials, I. Subsets of $\bbR$}, Invent.\ Math.\ \textbf{208} (2017), 217--245.

\bibitem{CSYZ2} J.~S.~Christiansen, B.~Simon, P.~Yuditskii, and M.~Zinchenko, \emph{Asymptotics of Chebyshev Polynomials, II.\ DCT subsets of $\bbR$}, Duke Math. J. \textbf{168} (2019), 325--349.



\bibitem{CSZ5} J.\ S.\ Christiansen, B.\ Simon, and M.\ Zinchenko, \emph{Asymptotics of Chebyshev Polynomials, V. Residual polynomials}, Ramanujan J.\ \textbf{61} (2023), 251--278.

\bibitem{CSZ6} J.\ S.\ Christiansen, B.\ Simon, and M.\ Zinchenko, \emph{Bounds for weighted Chebyshev and residual polynomials on subsets of $\bb R$}, Constr. Approx. \textbf{64} (2026), 87--104.

\bibitem{ELY24} B.\ Eichinger, M.\ Luki\'c, and G.\ Young, \emph{Asymptotics of Chebyshev rational functions with respect to subsets of the real line}, Constr. Approx. \textbf{59} (2024), 541--581.

\bibitem{EY11} A.\ Eremenko and P.\ Yuditskii,
    \emph{Comb functions}, in \emph{Recent advances in orthogonal polynomials, special functions, and their applications}, Contemp. Math., \textbf{578}, American Mathematical Society, Providence, RI, 2012, pp.~99--118.






\bibitem{GonHat15} A.\ Goncharov and B.\ Hatino\u{g}lu, \emph{Widom factors}, Potential Anal.\ \textbf{42} (2015), 671--680.















\bibitem{PehYud01} F.\ Peherstorfer and P.\ Yuditskii, \emph{Asymptotics of orthonormal polynomials in the presence of a  denumerable set of mass points}, Proc.\ Amer.\ Math.\ Soc. \textbf{129} (2001), 3213--3220.

\bibitem{PehYud03} F.\ Peherstorfer and P.\ Yuditskii, \emph{Asymptotic behavior of polynomials orthonormal on a homogeneous set}, J.\ Anal.\ Math.\ {\bf 89} (2003), 113--154.

\bibitem{PehYud06} F.\ Peherstorfer and P.\ Yuditskii, \emph{Remark on the paper ``Asymptotic behavior of polynomials orthonormal on a homogeneous set''}, arXiv:math/0611856.



\bibitem{Ran95} T.\ Ransford, \emph{Potential Theory in the Complex Plane}, Cambridge University Press, 1995.



\bibitem{Sch08} K.\ Schiefermayr, \emph{A lower bound for the minimum deviation of the Chebyshev polynomial on a compact real set}, East J. Approx. \textbf{14} (2008), 223--233.

\bibitem{Sch11} K.~Schiefermayr,
    \emph{A lower bound for the norm of the minimal residual polynomial}, Constr. Approx. \textbf{33} (2011), no.~3, 425--432.

\bibitem{Sch17} K.~Schiefermayr,
    \emph{The growth of polynomials outside of a compact set---the Bernstein--Walsh inequality revisited}, J. Approx. Theory \textbf{223} (2017), 9--18.





\bibitem{Sim11} B. Simon, \emph{Szeg\H{o}'s Theorem and Its Descendants: Spectral Theory for  $L^2$ Perturbations of Orthogonal Polynomials}, Princeton University Press, Princeton, 2011.

\bibitem{SimZla03} B.\ Simon and A.\ Zlato\v{s}, \emph{Sum rules and the Szeg\H{o} condition for orthogonal polynomials on the real line}, Comm.\ Math.\ Phys. \textbf{242} (2003), 393--423.

















\end{thebibliography}
\end{document}